\documentclass [11pt]{amsart}
\usepackage{amsmath}
\usepackage{amssymb}
\usepackage{amscd}
\usepackage{mathabx}
\usepackage{scalerel,stackengine}
\usepackage{bbm}
\usepackage{a4wide}
\usepackage{graphicx}

\allowdisplaybreaks

\stackMath
\newcommand\reallywidehat[1]{%
\savestack{\tmpbox}{\stretchto{%
  \scaleto{%
    \scalerel*[\widthof{\ensuremath{#1}}]{\kern-.6pt\bigwedge\kern-.6pt}%
    {\rule[-\textheight/2]{1ex}{\textheight}}
  }{\textheight}%
}{0.5ex}}%
\stackon[1pt]{#1}{\tmpbox}%
}

\newcommand\reallywidecheck[1]{%
\savestack{\tmpbox}{\stretchto{%
  \scaleto{%
    \scalerel*[\widthof{\ensuremath{#1}}]{\kern-.6pt\bigwedge\kern-.6pt}%
    {\rule[-\textheight/2]{1ex}{\textheight}}
  }{\textheight}%
}{0.5ex}}%
\stackon[1pt]{#1}{\scalebox{-1}{\tmpbox}}%
}

\allowdisplaybreaks[3]

\newcommand{\oplam}{\mbox{\LARGE $\curlywedge$}}

\numberwithin{equation}{section}

\newcommand{\supp}{\mbox{\rm supp}}

\newcommand{\RR}{{\mathbb R}}

\newcommand{\ZZ}{{\mathbb Z}}
\newcommand{\CC}{{\mathbb C}}
\newcommand{\TT}{\mathbb T}
\newcommand{\NN}{\mathbb N}
\newcommand{\KK}{\mathbb K}

\newcommand{\cL}{{\mathcal L}}

\newcommand{\cM}{{\mathcal M}}
\newcommand{\cS}{{\mathcal S}}

\newcommand{\dd}{\mbox{\rm d}}
\newcommand{\eps}{\varepsilon}

\newcommand{\Cu}{C_{\mathsf{u}}}
\newcommand{\Cc}{C_{\mathsf{c}}}
\newcommand{\Cz}{C^{}_{0}}

\newcommand{\LI}{\mathcal{LI}}
\newcommand{\DE}{\mathcal{DE}}

\theoremstyle{plain}
\newtheorem{theorem}{Theorem}[section]
\newtheorem{proposition}[theorem]{Proposition}
\newtheorem{lemma}[theorem]{Lemma}
\newtheorem{corollary}[theorem]{Corollary}
\newtheorem{fact}[theorem]{Fact}

\theoremstyle{definition}
\newtheorem{definition}[theorem]{Definition}
\newtheorem{remark}[theorem]{Remark}

\begin{document}
\title[Fourier transformable measures]{Fourier Transformable Measures with Meyer set support and their lift to the cut and project scheme}

\author{Nicolae Strungaru}
\address{Department of Mathematical Sciences, MacEwan University \\
10700 -- 104 Avenue, Edmonton, AB, T5J 4S2\\
and \\
Institute of Mathematics ``Simon Stoilow''\\
Bucharest, Romania}
\email{strungarun@macewan.ca}
\urladdr{http://academic.macewan.ca/strungarun/}

\begin{abstract} In this paper, we prove that given a cut-and-project scheme $(G, H, \mathcal{L})$ and a compact window $W \subseteq H$, the natural projection gives a bijection between the Fourier transformable measures on $G \times H$ supported inside the strip $\cL \cap (G \times W)$ and the Fourier transformable measures supported inside $\oplam(W)$, and relate their Fourier transforms. We use this formula to relate the Fourier transforms of the measures, and explain how one can use this relation to re-derive some known results about Fourier analysis of measures with Meyer set support.
\end{abstract}

\renewcommand{\thefootnote}{\fnsymbol{footnote}}
\renewcommand{\thefootnote}{\arabic{footnote}}

\maketitle

\section{Introduction}

After the discovery of quasi-crystals \cite{She}, it has become clear that we need to better understand the process of diffraction. Mathematically, the diffraction pattern of a solid can be viewed as the Fourier transform $\widehat{\gamma}$ of the autocorrelation measure $\gamma$ of the structure (see \cite{HOF3} for the setup and the monographs \cite{TAO,TAO2} for a general review of the theory). The measure $\gamma$ is positive definite and therefore it is Fourier transformable as a measure \cite{ARMA1,BF,MoSt} with positive Fourier transform $\widehat{\gamma}$. It is this measure $\widehat{\gamma}$ which models the diffraction of our solid.

Structures with pure point diffraction, that is structures for which $\widehat{\gamma}$ is a pure point measure are now very well understood. Building on the earlier work of Gil deLamadrid--Argabright \cite{ARMA}, Solomyak \cite{SOL,SOL1}, Lee--Moody--Solomyak \cite{LMS}, Baake--Moody \cite{BM}, Baake--Lenz \cite{BL}, Gouere \cite{Gouere-1,Gouere-2}, Moody--Strungaru \cite{MS}, Meyer \cite{Mey2}, pure point diffraction was characterized in \cite{LSS,LSS2}. The focus now shifted towards models with mixed diffraction spectrum, and especially those with a large pure point part.

The best mathematical models for Delone sets with a large pure point spectrum and (generic) positive entropy are Meyer sets. They have been introduced in the pioneering work of Meyer \cite{Meyer}, and popularized in the area of Aperiodic Order by Moody \cite{MOO,Moody} and Lagarias \cite{LAG,LAG1}. They are usually constructed via a cut and project scheme (or simply CPS) and can be characterised via harmonic analysis, discrete geometry, algebra and almost periodicity \cite{Meyer,Moody,NS11}. The basic idea behind a CPS is to project points from a higher dimensional lattice which lie within a bounded strip of the real space(see Def.~\ref{CPS} below for the exact definition). If the width of the strip (called the window) is regular, then the resulting model set is pure point diffractive \cite{Hof2,Martin2,BM,CRS2}. Recent work shown that this holds for a larger class of weak model sets \cite{BHS,KR,KR2,KRS,NS21c}.

As subsets of regular model sets, Meyer sets still exhibit a large pure point spectrum \cite{NS1,NS2,NS5,NS11,NS21,NS21b} and a highly ordered continuous spectrum \cite{NS1,NS5,NS21,NS21b}. The long range order of the spectrum of Meyer sets is inherited from a covering regular model set \cite{NS21,NS21b}, which in turn is a residue of the Poisson summation formula for the lattice in the CPS \cite{TAO,LO,CRS,CRS2}.

One would expect to be possible to related the diffraction of a Meyer set (or more generally a measure with Meyer set support) directly to the lattice $\cL$ in the CPS. It is the goal of this paper to establish this connection. Let us briefly explain our approach.

Fix a CPS $(G, H, \cL)$ and a compact set $W \subseteq H$. It is easy to see that
$$
\gamma= \sum_{x \in \oplam(W)} c_x \delta_x \qquad  \longleftrightarrow \qquad  \eta= \sum_{x \in \oplam(W)} c_x \delta_{(x,x^\star)}
$$
establishes a bijection between translation bounded measures supported inside $\oplam(W)$ and translation bounded measures supported inside $\cL \cap (G \times W)$.
We first show in Prop.~\ref{prop pd can be lifted} that $\gamma$ is positive definite if and only if $\eta$ is positive definite. Since each Fourier transformable measure supported inside a Meyer set can be written as a linear combination of positive definite measures supported inside a common model set, we use this to establish in Thm.~\ref{T1} that $\gamma$ is Fourier transformable if and only if $\eta$ is Fourier transformable, and relate their Fourier transform (see \eqref{eq6}).

We complete the paper by discussing in Section~\ref{sect discc} how these results can be used to re-derive the known properties of diffraction for measures with Meyer
set support, and potentially used to prove new results.

\section{Definitions and Notations}

Throughout the paper, $G$ denotes a second countable locally compact Abelian group (LCAG). By $\Cu(G)$ we denote the space of uniformly continuous and bounded functions on $G$. This is a Banach space with respect to the sup norm $\| . \|_\infty$. As usual, we denote by $\Cz(G)$ the subspace of $\Cu(G)$ consisting of functions vanishing at infinity, and by $\Cc(G)$ the subspace of compactly supported continuous functions. Note that $\Cc(G)$ is not complete in $(\Cu(G), \| . \|_\infty)$.

In the spirit of \cite{ARMA} we denote by
\begin{displaymath}
K_2(G) := \mbox{Span} \{ \varphi *\psi : \varphi,\psi \in \Cc(G) \} \,.
\end{displaymath}

\smallskip

Given two LCAG's $G$ and $H$ and two functions $g : G \to \CC, h : H \to \CC$, we denote by $g \odot h : G \times H \to \CC$ their tensor product
\begin{displaymath}
( g \odot h) (x,y)= g(x) h(y) \,.
\end{displaymath}
It is obvious that whenever $\varphi \in \Cc(G), \psi \in \Cc(H)$ we have $\varphi \odot \psi \in \Cc(G \times H)$. Moreover, if $\varphi \in K_2(G)$ and $\psi \in K_2(H)$, we have $\varphi \odot \psi \in K_2(G \times H)$.

In the rest of this section, we review some of the basic concepts which are important for this paper. For a more general review of these we recommend \cite{TAO,TAO2}.

\subsection{Measures}

In the spirit of Bourbacki \cite{BOURB}, by a measure we understand a linear functional on $\Cc(G)$ which is continuous with respect to the inductive topology. This notion corresponds to the classical concept of a Radon measure (see \cite[Appendix]{CRS2}). For the case $G=\RR^d$ a clear exposition of this is given in \cite{TAO}.

\begin{definition} A linear functional $\mu : \Cc(G) \to \CC$ is called a \textbf{Radon measure} if for each compact set $K \subseteq G$ there exists a constant $C_K$ such that, for all $\varphi \in \Cc(G)$ with $\supp(\varphi) \subseteq K$ we have
\begin{displaymath}
\left| \mu(\varphi) \right| \leq C_K \| \varphi \|_\infty \,.
\end{displaymath}
We will often write $\int_G \varphi(t) d \mu(t)$ instead of $\mu(\varphi)$.

$\mu$ is called \textbf{positive} if for all $\varphi \in \Cc(G)$ with $\varphi \geq 0$ we have $\mu(\varphi) \geq 0$.
\end{definition}

By the Riesz representation Theorem \cite{Rud}, a positive Radon measure is simply a positive regular Borel measure. Moreover, each Radon measure is a linear combination of (at most four) positive Radon measures \cite[Appendix]{CRS2}.

Next, we review the total variation of a measure.

\begin{definition}
Given a measure $\mu$, we can define \cite{Ped,ReiterSte,CRS2} a positive measure $\left| \mu \right|$, called the \textbf{total variation} of $\mu$, such that for all $\varphi \in \Cc(G)$ with $\varphi \geq 0$ we have
\begin{displaymath}
 \left| \mu \right| (\varphi) = \sup \{ \left|\mu(\psi) \right| : \psi \in \Cc(G), \mbox{ with } |\psi| \leq \varphi \} \,.
\end{displaymath}
\end{definition}

We are now ready to introduce the concept of translation boundedness for measures and norm almost periodicity.

\begin{definition} Let $A \subseteq G$ be a fixed pre-compact set with non-empty interior. We define the \textbf{ $A$-norm} of $\mu$ via
\begin{displaymath}
\| \mu \|_A := \sup_{x \in G} \left| \mu \right| (x+A) \,.
\end{displaymath}

A measure $\mu$ is called \textbf{translation bounded} if $\| \mu \|_A <\infty$.
\end{definition}

\begin{remark}\cite{BM,SS} Different pre-compact sets $A_1, A_2$ with non-empty interior define equivalent norms. Therefore, the definition of translation boundedness does not depend on the choice of $A$.

This allows us to define
\begin{displaymath}
\cM^\infty(G) := \{ \mu : \mu \mbox{ is a translation bounded measure } \} \,.
\end{displaymath}
Then $(\cM^\infty(G), \| . \|_A)$ is a normed space. It is in fact a Banach space \cite{CRS3}.
\end{remark}

\bigskip

Next we review the definition of norm almost periodicity as introduced in \cite{BM}.

\begin{definition} Let $ A \subseteq G$ be a fixed pre-compact set with non-empty interior. A measure $\mu \in \cM^\infty(G)$ is called \textbf{ norm almost periodic} if for each $\eps >0$ the set
\begin{displaymath}
P_\eps^A (\mu) := \{ t \in G : \| T^t \mu -\mu \|_A < \eps \} \,,
\end{displaymath}
of \textbf{$\eps$-norm almost periods} of $\mu$ is relatively dense.
\end{definition}

As discussed above, different pre-compact sets define equivalent norms. This means that while the set of $\eps$-norm almost periods on $\mu$ depends on the choice of $A$, the almost periodicity of $\mu$ is not dependent of this choice.

Any norm almost periodic measure is strongly almost periodic \cite{BM}, and the two concepts are equivalent for measures with Meyer set support \cite{BM}. In general, norm almost periodicity is an uniform version of strong almost periodicity \cite[Thm.~4.7]{SS}. The class of norm almost periodic pure point measure was studied in detail and characterized in \cite{NS11}.

\smallskip
Let us next recall positive definiteness for functions and measures. For more details, we recommend \cite{BF,MoSt}.

\begin{definition}\label{def PD funct} A function $f : G\longrightarrow \CC$ is called \textbf{positive definite} if, for all $ n \in \NN$ and all $x_1,\ldots, x_n\in G$, the
matrix $\left(f (x_k - x_l)\right)_{k,l=1,\ldots, n}$ is positive Hermitian. This is equivalent to
$$
\sum_{k,l=1}^n \overline{c_l} f (x_k - x_l) c_k \geq 0 \qquad \forall  n \in \NN, x_1,\ldots, x_n\in G, c_1, \ldots , c_n \in \CC
$$

A measure $\mu$ is called \textbf{positive definite} if for all $ \varphi \in \Cc(G)$ we have
\begin{displaymath}
\mu (\varphi * \widetilde{\varphi}) \rangle \geq 0 \,.
\end{displaymath}
\end{definition}
This is equivalent to $\mu*\varphi*\tilde{\varphi}$ being a positive definite function for all $\varphi \in \Cc(G)$ \cite{BF,MoSt}.
\smallskip

We complete the subsection by reviewing the notion of Fourier transformability for measures. For a more detailed review of the subject, we recommend \cite{MoSt}.

\begin{definition} A measure $\mu$ on $G$ is called \textbf{Fourier transformable} if there exists a measure $\widehat{\mu}$ on $\widehat{G}$ such that, for all $\varphi \in K_2(G)$ we have $|\check{\varphi}| \in L^1( |\widehat{\mu} |)$ and
\begin{displaymath}
\int_{G} \varphi(t) \dd \mu(t)  = \int_{\widehat{G}} \widecheck{\varphi}(\chi) \dd \widehat{\mu}(\chi) \,.
\end{displaymath}
\end{definition}

\subsection{Cut and Project Schemes and Meyer sets}

In this part, we review some notions related to the cut and project formalism. For more details, we recommend \cite{TAO,MOO,Moody}.

\begin{definition}\label{CPS} By a \textbf{cut and project scheme} (or simply \textbf{CPS}) we understand a triple $(G, H, \cL)$ consisting of a $\sigma$-compact LCAG $G$, a LCAG $H$, and a lattice $\cL \subset G \times H$ such that
\begin{itemize}
  \item [(i)] $\pi_H(\cL)$ is dense in $H$.
  \item [(ii)] the restriction $\pi_G|_\cL$ of the first projection $\pi_G$ to $\cL$ is one to one.
\end{itemize}
\end{definition}

Given a CPS $(G,H, \cL)$ we will denote by $L:= \pi_G(\cL)$. Then, $\pi_G$ induces a group isomorphism between $\cL$ and $L$. Composing the inverse of this with the second projection $\pi_H$ we get a mapping
\begin{displaymath}
\star : L \to H
\end{displaymath}
which we will call the \textbf{$\star$-mapping}. We then have
\begin{displaymath}
\cL = \{ (x,x^\star) : x \in L \} \,.
\end{displaymath}

\smallskip

Given a CPS $(G, H, \cL)$ and a subset $W \subseteq H$ we can define
$$
\oplam(W):= \{ x \in L : x^\star \in W \} \,.
$$
When $W$ is compact, we will call $\oplam(W)$ a \textbf{weak model set}. If furthermore $W$ has non-empty interior $\oplam(W)$ is called a \textbf{model set}.

\smallskip

Next, let us review the concept of a Meyer set, which plays a fundamental role in this paper.

\begin{definition} A set $\Lambda \subset G$ is called a \textbf{Meyer set} if $\Lambda$ is relatively dense and $\Lambda-\Lambda-\Lambda$ is uniformly discrete.
\end{definition}

For equivalent characterisations of Meyer sets see \cite{LAG,LAG1,Meyer,MOO,NS11}. Of importance to us will be the following result.

\begin{theorem}\label{thm: meyer CPS}\cite{NS11} Let $\Lambda \subseteq G$ be relatively dense. Then $\Lambda$ is Meyer if and only if it is a subset of a (weak) model set.

Moreover, if $\Lambda$ is Meyer, it is a subset of a weak model set in a CPS $(G,H, \cL)$ with metrisable and compactly generated $H$.
\end{theorem}

\smallskip

Given a CPS $(G,H, \cL)$, the map
\begin{equation}\label{eq77}
L \ni  x \to (x,x^\star) \in \cL \,.
\end{equation}
is a group isomorphism, and hence it induces an isomorphism between the spaces of (bounded) functions on $L$ and $\cL$, respectively. Since $\cL$ is a discrete group, the space of (translation) bounded measures on $\cL$ can be identified with the space of (bounded) functions on $\cL$. On another hand $L$ is typically dense in $G$, and many functions on $L$ do not induce pure point measures.

For us, of interest will be measures supported inside weak model sets $\oplam(W)$. Since $\oplam(W)$ is uniformly discrete \cite{MOO}, the space of (translation) bounded measures on $\oplam(W)$ can be identified with the space of (bounded) functions on, and corresponds via the above isomorphism with the spaces of (translation) bounded measures or (bounded) functions on $\cL$, respectively, that are supported inside $G \times W$.

Our focus in this paper is on these two spaces. We will study them as spaces of measures, and we will be interested in the relation between the Fourier theory of these two spaces, and the behaviour of the Fourier transform
with respect to the isomorphism induced by \eqref{eq77}. For this reason, let us introduce the following notations.

Given a CPS $(G,H, \cL)$ and a compact set $W$ we denote by
\begin{align*}
\cM^\infty(\oplam(W)) &:= \{ \mu \in \cM^\infty(G) : \supp(\mu) \subseteq \oplam(W)\}  \\
\cM^\infty(\cL; W) &:=\{ \nu \in \cM^\infty(G \times H) : \supp(\nu) \subseteq \left( \cL \cap (G \times W) \right) \}  \,.
\end{align*}
The isomorphism \eqref{eq77} induces a bijection $\LI_{G,H,\cL, W}: \cM^\infty(\oplam(W)) \to   \cM^\infty(\cL ; W)$
$$
\LI_{G,H,\cL, W}(\mu)= \sum_{(x,x^\star) \in \cL} \mu(\{x \}) \delta_{(x,x^\star)} \,,
$$
with inverse $\DE_{G,H,\cL, W}: \cM^\infty(\cL ;W) \to \cM^\infty(\oplam(W))$
$$
\DE_{G,H,\cL, W}(\nu)= \sum_{(x,x^\star) \in \cL} \nu(\{(x,x^\star) \}) \delta_{x} \,.
$$
We will refer to these mappings as the \textbf{lift operator} and the \textbf{descent operator }, respectively.

The main results in this paper are that these operators are bijections between the subspaces of Fourier transformable (or cones of positive definite) measures, and relate their Fourier transform.

To understand the connection between the Fourier transforms, let us recall the notion of dual CPS. Given a CPS $(G, H, \cL)$, we can define
\begin{displaymath}
\cL^0 := \{ (\chi, \psi) \in \widehat{G} \times \widehat{H} : \chi(x)\psi(x^\star) =1 \, \forall x \in L \} \,.
\end{displaymath}
Then, $(\widehat{G}, \widehat{H}, \cL^0)$ is a CPS \cite{BHS,MOO,Moody,NS11}. We will refer to this as the CPS \textbf{dual to $(G, H, \cL)$}.

\section{Positive definite measures with Meyer set support}

In this section, we show that $\LI_{G,H,\cL, W}$ and $\DE_{G,H,\cL, W}$ take positive definite measures to positive definite measures.

\smallskip
Let us start with the following obvious Lemma, which follows immediately from Def.~\ref{def PD funct} and the fact that the function from \eqref{eq77} is a group isomorphism.

\begin{lemma}\label{L1} Let $(G, H, \cL)$ be a CPS and let $f : L \to \CC$ be a function. Define $g :\cL \to \CC$ via
$$
g(x,x^\star):= f(x) \,.
$$
Then $f$ is positive definite on $L$ if and only if $g$ is positive definite on $\cL$. \qed
\end{lemma}

\medskip
Next, we prove a slight generalisation of \cite[Lemma~2.10]{LS1} and \cite[Lemma~8.4]{TAO}.

\begin{lemma}\label{restriction is pd} Let $\gamma$ be a positive definite pure point measure on $G$, and let $L$ be any subgroup of $G$. Then, the function $g : L \to \CC$ defined via
\begin{displaymath}
g(x):= \gamma ( \{ x \}) \,,
\end{displaymath}
is a positive definite function on $L$.
\end{lemma}
\begin{proof}
Define $f : G \to \CC$ via
\begin{displaymath}
f(x):= \gamma ( \{ x \}) \,.
\end{displaymath}
Then $f$ is a positive definite function on $G$ \cite[Prop~2.4]{LS1}. Then, Def.~\ref{def PD funct} immediately gives that the restriction $g=f|_L$ to the subgroup $L$ is a positive definite function on $L$.
\end{proof}

We will also need the following result.

\begin{lemma}\label{extension is pd} Let $G$ be any group and let $H \leq G$ be a subgroup. Let $f : H \to \CC$ be a positive definite functions. Then, the function $g : G \to \CC$ defined via
\begin{displaymath}
g(x):= \left\{
\begin{array}{l c}
f(x) & \mbox{ if } x \in H \\
0 & \mbox{ otherwise }
\end{array}
 \right.
\end{displaymath}
is positive definite on $G$.
\end{lemma}
\begin{proof}
Let $n \in \NN, x_1,\ldots, x_n\in G$ and $c_1,..,c_n \in \CC$. Note that $g(x_k-x_l) =0$ whenever $x_k -x_l \notin H$.

On $G$ define the standard equivalence $\pmod{H}$ as
\begin{displaymath}
x \equiv y \pmod{H} \Leftrightarrow x-y \in H \,.
\end{displaymath}
This induces an equivalence relation on the set $\{ x_1,..., x_n \}$, and hence we can partition this set in equivalence classes $F_1,..., F_m$.

To make the computation clearer, define $c : G\to \CC$
$$
c(x):=
\left\{
\begin{array}{cc}
  c_j & \mbox{ if } x=x_j \\
  0 & \mbox{ otherwise }
\end{array}
\right.\,.
$$
Then,
\begin{align*}
\sum_{k,l=1}^n g(x_k-x_l)c_k \overline{c_l} &=\sum_{k,l=1}^n g(x_k-x_l)c(x_k) \overline{c(x_l)} \\
&=
 \sum_{i=1}^m \sum_{j=1}^m \left( \sum_{x \in F_i} \sum_{y \in F_j}  g(x-y) c(x) \overline{c(y)} \right)\\
&= \sum_{i=1}^m  \left( \sum_{x,y \in F_i}  g(x-y)c(x)\overline{c(y)} \right) \,.
\end{align*}
Now, fix some $1 \leq i \leq m$, and let $F_i:=\{ z_1,.., z_q \}$. Then
\begin{align*}
\sum_{x,y \in F_i}  g(x-y)c(x)\overline{c(y)}&= \sum_{r,s=0}^q g(z_r-z_s)c(z_r)\overline{c(z_s)} \\
&= \sum_{r,s=1}^q g(z_r-z_s)c(z_r)\overline{c(z_s)} \\
&= \sum_{r,s=1}^q f\left((z_r-z_1) -(z_s-z_1)\right)c(z_r)\overline{c(z_s)}  \geq 0 \\
\end{align*}
by the positive definiteness of $f$ applied to $m; y_1:=z_1 - z_1 ; y_2:=z_2 - z_1 ;...; y_q:=z_q - z_1 \in H$ and $c'_1=c(z_1),...,c'_q=c(z_q)$.

Therefore, for each $i$ we have $\sum_{x,y \in C_i}  g(x-y)c(x)\overline{c(y)} \geq 0$, and hence
\begin{displaymath}
\sum_{k,l=1}^n g(x_k-x_l)c_k \overline{c_l}=\sum_{i=1}^m  \left( \sum_{x,y \in C_i}  g(x-y)c(x)\overline{c(y)} \right)  \geq 0 \,.
\end{displaymath}
\end{proof}

\begin{remark} One can also prove Lemma~\ref{extension is pd} by using Fourier analysis. Indeed, since $f$ is positive definite, the measure $\mu:=f \theta_{H_{\mathsf{d}}}$ is a positive definite measure on the discrete group $H_{\mathsf{d}}$ \cite[Cor.~4.3]{ARMA1}. Then, it is Fourier transformable on $H_{\mathsf{d}}$ and its Fourier transform is positive \cite{ARMA1,BF}. As $H_{\mathsf{d}}$ is closed in the discrete group $G_{\mathsf{d}}$, by \cite[Thm.~4.2]{ARMA1}, the measure $\nu:=g \theta_{G_{\mathsf{d}}}$ is Fourier transformable on $G_{\mathsf{d}}$ and has positive Fourier transform. Then, $\mu$ is positive definite \cite[Thm.~4.1]{ARMA1}. Therefore, by \cite[Prop.~2.4]{LS1}, $g$ is positive definite on $G$.
\end{remark}

We are now ready to prove the following result.

\begin{proposition}\label{prop pd can be lifted} Let $(G,H, \cL)$ be a CPS, $\oplam(W)$ be a weak model set, and $f : G \to \CC$ be a function which vanishes outside $\oplam(W)$.
Let
\begin{align*}
\gamma &=\sum_{x \in \oplam(W)} f(x) \delta_x \\
\eta&= \sum_{(x,x^\star) \in \cL}  f(x) \delta_{(x,x^\star)}= \LI_{G, H, L, W} (\gamma) \,.
\end{align*}
Then $\gamma$ is a positive definite measure on $G$ if and only if $\eta= \LI_{G,H,\cL, W}(\gamma)$ is a positive definite measure on $G \times H$.
\end{proposition}

\begin{proof}

$\Rightarrow$: Denote as usual $L := \pi_G(\cL)$. Define $g: L \to \CC$ via
\begin{displaymath}
g(x):= \gamma ( \{ x \}) \,,
\end{displaymath}
that is $g= f|_{L}$.

Then, by Lemma~\ref{restriction is pd}, $g$ is a positive definite function on $L$ and hence, by Lemma~\ref{L1}, the function $h: \cL \to \CC$
\begin{displaymath}
h(x,x^\star)=g(x)
\end{displaymath}
is a positive definite function on $\cL$. Therefore, by Lemma~\ref{extension is pd}, the function $j: G \times H \to \CC$
\begin{displaymath}
j(z):= \left\{
\begin{array}{l c}
h(z) & \mbox{ if } z \in \cL \\
0 & \mbox{ otherwise }
\end{array}
 \right.
\end{displaymath}
is positive definite on $G \times H$. The claim follows from \cite[Prop.~2.4]{LS1}.

$\Leftarrow:$ Since $\eta$ is positive definite, by Lemma~\ref{restriction is pd} the function $h: \cL \to \CC$ defined by
\begin{displaymath}
h(x,x^\star):= \eta(\{(x,x^\star) \}) = f(x) \,,
\end{displaymath}
is positive definite on $\cL$ and hence, by Lemma~\ref{L1}, the restriction $g=f|_L$ is positive definite on $L$. As $f$ is zero outside $\oplam(W) \subseteq L$, it follows from Lemma~\ref{extension is pd} that $f$ is a positive definite function on $G$. The claim follows now from \cite[Prop.~2.4]{LS1}.
\end{proof}

\begin{remark}
\begin{itemize}
  \item[(a)] In Proposition~\ref{prop pd can be lifted}, the positive definiteness of $\eta$ and $\gamma$ is equivalent to the positive definiteness of the function $f$.
  \item[(b)] Denoting by
  \begin{align*}
\mathcal{PD}(\oplam(W)) &:= \{ \mu \in \cM^\infty(\oplam(W)) : \mu \mbox{ is positive definite } \}  \\
\mathcal{PD}(\cL; W) &:=\{ \nu \in \cM^\infty(\cL, W) : \nu \mbox{ is positive definite } \}  \,,
\end{align*}
Prop.~\ref{prop pd can be lifted} says that
\begin{align*}
  \LI_{G,H,\cL, W}(\mu)( \mathcal{PD}(\oplam(W)) ) &=\mathcal{PD}(\cL; W)  \\
  \DE_{G,H,\cL, W}(\mu) (\mathcal{PD}(\cL; W)  )&= \mathcal{PD}(\oplam(W))  \,.
\end{align*}
\end{itemize}

\end{remark}

\smallskip

\section{The lift of Fourier transformable measures}

We can now prove that, given a CPS $(G, H,\cL)$ and a compact set $K$, the lifting operator induces a bijection between the space of Fourier transformable measures supported inside $\oplam(W)$ and the space of Fourier transformable measures supported inside $ \cL \cap (G \times W)$.

\begin{theorem}\label{T1} Let $(G,H, \cL)$ be a CPS and let $W \subseteq H$ be compact. Let $\gamma$ be a translation bounded measure supported inside $\oplam(W)$ and let
$$
\eta:= \LI_{G,H,\cL, W}(\gamma) \,.
$$
Then $\gamma$ is Fourier transformable if and only if $\eta$ is Fourier transformable.

Moreover, if $\varphi \in K_2(H)$ is any function so that $\varphi \equiv 1$ on $W$, then for all $\psi \in \Cc(\widehat{G})$ we have $\psi \otimes \hat{\varphi} \in L^1(\eta)$ and
\begin{equation}\label{eq6}
\widehat{\gamma}(\psi)= \widehat{\eta}( \psi \otimes \check{\varphi})=: (\widehat{\eta})_{\check{\varphi}}(\psi) \,.
\end{equation}
\end{theorem}
\begin{proof}
$\Longrightarrow$ By \cite[Lemma~8.3]{NS21}, there exists a compact set $W \subseteq K$ and four positive definite measures $\omega_{1},\omega_{2},\omega_{3},\omega_{4}$ supported inside $\oplam(K)$ such that
$$
\gamma=\omega_1-\omega_2+i\omega_3-i \omega_4 \,.
$$
Then, we have
\begin{align*}
\eta&= \LI_{G,H,\cL, W}(\gamma) =\LI_{G,H,\cL, K}(\gamma)=\LI_{G,H,\cL, K}(\omega_1-\omega_2+i\omega_3-i \omega_4 )\\
 &= \LI_{G,H,\cL, K}(\omega_1)-\LI_{G,H,\cL, K}(\omega_2)+i\LI_{G,H,\cL, K}(\omega_3)-i\LI_{G,H,\cL, K}(\omega_4) \,.
\end{align*}
Now, by Prop.~\ref{prop pd can be lifted}, for all $1 \leq j \leq 4$ the measure $\LI_{G,H,\cL, K}(\omega_j)$ is positive definite. Therefore, as a linear combination
of positive definite measures, $\eta$ is Fourier transformable.

\medskip
$\Longrightarrow$. Our argument is similar to \cite{CRS}. We split the argument into two steps.

\emph{Step 1:} We show that $(\widehat{\eta})_{\check{\varphi}}$ is a measure.

Let us first note that for all $\psi \in K_2(G)$ we have $\psi \otimes \varphi \in K_2(G \times H)$. Therefore, since $\eta$ is Fourier transformable, we have
\begin{equation}\label{eq2}
\left| \check{\psi} \otimes  \check{\varphi} \right| \in L^1(|\widehat{\eta}|) \,.
\end{equation}

We now show that for all $\phi \in \Cc(\widehat{G})$ we have $\left| \phi \otimes  \check{\varphi} \right| \in L^1(|\widehat{\eta}|)$ and that
$$
(\widehat{\eta})_{\check{\varphi}}(\phi):= \widehat{\eta}( \phi \otimes \check{\varphi}) \,.
$$
defines a measure.

Let $K \subseteq \widehat{G}$ be a fixed compact set. Then, there exists some $\psi \in K_2(G)$ such that $\check{\psi} \geq 1_K$ \cite{BF,MoSt}.

Now, for all $\psi \in \Cc(\widehat{G})$ with $\supp(\psi) \subseteq K$ we have
\begin{equation}\label{eq1}
|\widehat{\eta}| \left(|\phi \otimes \check{\varphi}| \right)=\int_{G \times H} \left| \phi(s)\right| \cdot  \left|  \check{\varphi}(t) \right| \dd |\widehat{\eta}|(s,t) \leq \| \psi \|_\infty  \int_{G \times H} \left| \check{\psi}(s)\right| \cdot  \left|  \widehat{\varphi}(t) \right| \dd |\widehat{\eta}|(s,t) < \infty
\end{equation}
and hence $(\widehat{\eta})_{\check{\varphi}}$ is well defined.

Moreover, for all $\psi \in \Cc(\widehat{G})$ with $\supp(\psi) \subseteq K$, it follows from \eqref{eq1} that
$$
\left| (\widehat{\eta})_{\check{\varphi}}(\psi) \right| \leq   C_K \| \psi \|_\infty
$$
where
$$
C_K:= \int_{G \times H} \left| \widehat{\psi}(s)\right| \cdot  \left|  \widehat{\varphi}(t) \right| \dd |\widehat{\eta}|(s,t)  < \infty \,.
$$
This shows that $(\widehat{\eta})_{\check{\varphi}}$ is a measure.

\medskip

\emph{Step 2:} We show that for all $\psi \in K_2(G)$ we have $|\widecheck{\psi}| \in L^1(|(\widehat{\eta})_{\check{\varphi}}|)$ and
$$
(\widehat{\eta})_{\check{\varphi}} (\widecheck{\psi}) = \gamma(\psi) \,.
$$

Let $\psi \in K_2(G)$ be arbitrary.

Since $G$ is second countable, so is $\widehat{G}$ \cite{ReiterSte}. In particular $\widehat{G}$ is $\sigma$-compact \cite{ReiterSte}. Therefore, there exists a sequence $K_n$ of compact sets with
$K_n \subseteq (K_{n+1})^\circ$ such that
$$
\widehat{G}= \bigcup_{n} K_n \,.
$$
Let $\phi_n \in \Cc(\widehat{G})$ be so that $1_{K_n} \leq \phi_n \leq 1_{K_{n+1}}$.

Then, $\phi_n \widehat{\psi} \in \Cc(\widehat{G})$ and by the definition of $(\widecheck{\eta})_{\check{\varphi}}$ we have
$$
( \widehat{\eta})_{\check{\varphi}}(\phi_n \widecheck{\psi}) = \widehat{\eta}\left( ( \phi_n \widecheck{\psi}) \otimes \widecheck{\varphi} \right) \,.
$$
Now, for all $n$ we have by \eqref{eq2}
$$
\left| ( \phi_n \widecheck{\psi}) \otimes \widecheck{\varphi}  \right| \leq \left| \widecheck{\psi}\right| \otimes  \left| \widecheck{\varphi} \right| \in L^1(|\widehat{\eta}|)
$$
Therefore, by the Dominated convergence theorem  (\cite[Thm.~3.2.51]{ReiterSte}), we have
\begin{align}
\widehat{\eta}\left( ( \widecheck{\psi}) \otimes \widecheck{\varphi} \right)  &= \lim_n  \widehat{\eta} \left(( \phi_n \widecheck{\psi}) \otimes \widecheck{\varphi} \right) = \lim_n ( \widehat{\eta})_{\check{\varphi}} \left( \phi_n \widecheck{\psi} \right)  \label{eq3}\,.
\end{align}

Next, by the monotone convergence theorem \cite{ReiterSte} we have
$$
| (\widehat{\eta})_{\check{\varphi}}| \left( |\widecheck{\psi}| \right) = \lim_n  | (\widehat{\eta})_{\check{\varphi}}| \left( |\phi_n \widecheck{\psi}| \right) \,.
$$

Note that for each $n$ we have
 \begin{align*}
 | (\widehat{\eta})_{\check{\varphi}}| \left( |\phi_n \widecheck{\psi}| \right) &= \sup \{ \left|  (\widehat{\eta})_{\check{\varphi}}( \Psi)  \right| : \Psi \in \Cc(\hat{G}) , |\Psi| \leq |\phi_n \widecheck{\psi}|  \}\\
  &=\sup \{ \left| \widehat{\eta} \left( \Psi \otimes \check{\varphi} \right) \right| : \Psi \in \Cc(\hat{G}) , |\Psi| \leq |\phi_n \widecheck{\psi}|  \} \\
  &= \sup \{ \left| \int_{\hat{G} \times \hat{H}} \Psi(x) \check{\varphi}(y) \dd  \widehat{\eta}(x,y) \ \right| : \Psi \in \Cc(\hat{G}) , |\Psi| \leq |\phi_n \widecheck{\psi}|  \}\\
  &\leq   \sup \{ \int_{\hat{G} \times \hat{H}} |\Psi(x) \check{\varphi}(y)| \dd  |\widehat{\eta}(x,y)|  : \Psi \in \Cc(\hat{G}) , |\Psi| \leq |\phi_n \widecheck{\psi}|  \}\\
   &\leq  \int_{\hat{G} \times \hat{H}} |\phi_n \widecheck{\psi}(x)   \check{\varphi}(y)| \dd  |\widehat{\eta}(x,y)| \leq  \int_{\hat{G} \times \hat{H}} |\widecheck{\psi}(x)   \check{\varphi}(y)| \dd  |\widehat{\eta}(x,y)|  \\
   &=|\widehat{\eta} | \left(| \widecheck{\psi} \otimes  \check{\varphi}| \right) \,.
 \end{align*}
Since $\widecheck{\psi} \otimes   \check{\varphi} \in L^1(|\widehat{\eta}|)$ we get
$$
|(\widehat{\eta})_{\check{\varphi}}|\left( |\widecheck{\psi}| \right) \leq |\widehat{\eta} | \left(| \widecheck{\psi} \otimes  \check{\varphi}| \right)  < \infty \,.
$$
This shows that $ |\widecheck{\psi}| \in L^1( | (\widehat{\eta})_{\check{\varphi}}|)$. Therefore, $\phi_n \widecheck{\psi}$ is dominated by  $ |\widecheck{\psi}| \in L^1( |( \widehat{\eta})_{\check{\varphi}}|)$ and converges pointwise to $\widecheck{\psi}$. Thus, by \eqref{eq3} and the dominated convergence theorem we get
\begin{align*}
 \widehat{\eta}\left( (\widecheck{\psi}) \otimes \widecheck{\varphi} \right)
   &= \lim_n ( \widehat{\eta})_{\check{\varphi}} \left( \phi_n \widecheck{\psi} \right)  = (\widehat{\eta})_{\check{\varphi}} \left( \widecheck{\psi} \right) \,.
\end{align*}
Finally, by the Fourier transformability of $\eta$ we have
$$
\widehat{\eta}\left( \widecheck{\psi} \otimes \widecheck{\varphi} \right)= \eta \left( \psi \otimes \varphi \right) =\gamma(\psi) \,.
$$
Therefore, we proved that for all $\psi \in K_2(G)$ we have $\widecheck{\psi} \in L^1(|(\widehat{\eta})_{\check{\varphi}}|)$ and
$$
( \widehat{\eta})_{\check{\varphi}} \left( \widecheck{\psi} \right) = \gamma(\psi) \,.
$$
This proves that $\gamma$ is Fourier transformable and
$$
\widehat{\gamma}= (\widehat{\eta})_{\check{\varphi}} \,,
$$
completing the proof.

\end{proof}

\section{Applications}\label{sect discc}

In this section we will discuss the relation \eqref{eq6} and how can it be used to (re)derive some results from \cite{NS21}.

To make the things easier to follow we introduce the notion of strongly admissible functions.

\subsection{Strongly admissible functions for CPS.}

Let us start with the following definition.

\begin{definition} Given a group $H$ of the form $H=\RR^d \times H_0$, with a LCAG $H_0$, a function $h : H \to \CC$ is called
\textbf{strongly admissible} if there exists $f\in \Cu(\RR^d)$ and $\varphi \in \Cc(H_0)$ such that
\begin{itemize}
  \item{} $\| (1+|x|^{2d}) f \|_\infty < \infty$;
  \item{} $h= f \otimes \varphi$.
\end{itemize}
\end{definition}

Next, given a CPS $(G,H,\cL)$, we will denote by $\cM_{\cL}(G \times H)$ the space of $\cL$-periodic measures on $G \times H$. Note that by \cite[Prop.~6.1]{LR}
$$
\cM_{\cL}(G \times H) \subseteq \cM^\infty(G \times H) \,.
$$

\smallskip
We will see below that given a Fourier transformable measure $\gamma$ supported inside Meyer set $\Lambda$, Theorem~\ref{T1} can be used to create a CPS $(G, H=\RR^d \times H_0, \cL)$, a $\cL^0$ periodic measure $\rho(=\widehat{\eta})$ and a strongly admissible function $h$ on $\widehat{H} =\RR^d \times \widehat{H_0}$ such that, \eqref{eq6} gives
$$
\gamma= (\rho)_{h} \,.
$$
This motivates us to closely look at the properties of $(\rho)_h$, where $\cM_{\cL}(G \times H)$ for a CPS $(G, H=\RR^d \times H_0, \cL)$ and a strongly admissible function $h$.
While for applications this will be the dual CPS, we will study in general CPS.

\smallskip
Let us start with the following simple observation which also explains the name "strongly admissible".

Given a CPS $(G, H=\RR^d \times H_0, \cL)$, a measure $\rho \in \cM_{\cL}(G \times H)$ and strongly admissible function $h$, it is obvious that the function $h$ is admissible for $(G, H, \cL,\rho)$ in the sense of \cite[Def.~3.1]{LR}. Therefore, by \cite[Prop.~6.3]{LR}, we can define a translation bounded measure $\rho_h$ on $G$ via
$$
\rho_h(\phi):= \rho (\phi \otimes h) \qquad \forall \phi \in \Cc(G) \,.
$$
This measure is strongly almost periodic by \cite[Thm.~3.1]{LR}. In fact, the strong admissibility of $h$ immediately implies that $\rho_h$ is norm almost periodic.

Indeed, let $(G, H=\RR^d \times H_0, \cL)$, $h=f \otimes \varphi$ be strongly admissible and $\rho \in \cM_{\cL}(G \times H)$.
Let $W \subseteq \widehat{H_0}$ be any compact set containing $\supp (\varphi)$ and let $K, K_1 \subseteq G$ be compact in $G$ with $K \subseteq K_1^\circ$. Then, for all compact sets a standard computation similar to \cite[Lemma~5.2]{NS21} shows that
$$
\| (\rho)_h \|_{K} \leq C \|  \varphi \|_\infty \| (1+|x|^{2d}) f \|_\infty \| \rho \|_{K_1 \times [-\frac{1}{2},\frac{1}{2}]^d \times W } \,,
$$
where
\begin{displaymath}
C:= \left( \sum_{n \in \ZZ^d}  \sup_{z \in n+[-\frac{1}{2},\frac{1}{2}]^d } \frac{1}{1+|z|^{2d}} \right)  < \infty \,.
\end{displaymath}
This immediately gives the following stronger version of \cite[Lemma~5.2]{NS21} .
\begin{fact}\label{F2} Let $(G, H=\RR^d \times H_0, \cL)$ be a CPS, $\rho \in \cM_{\cL}(G \times H)$ and $h \in \Cz(H)$ be strongly admissible. Then $\rho_h$ is a
norm almost periodic measure.
\end{fact}

\subsection{Fourier transform of measures with Meyer set support}\label{subs 1}

Fix an arbitrary Meyer set $\Lambda$ and a Fourier transformable measure $\gamma$ with $\supp(\gamma) \subset \Lambda$.

By Theorem~\ref{thm: meyer CPS} and the structure theorem of compactly generated groups, there exists a CPS $(G, \RR^d \times \ZZ^m \times \KK, \cL)$ with compact $\KK$ and a compact $W \subseteq \RR^d \times \ZZ^m \times \KK$ such that
$$
\Lambda \subseteq \oplam(W) \,.
$$
By eventually enlarging $W$ we can assume without loss of generality that
$$
W=W_0 \times F \times \KK
$$
for compact $W_0 \subseteq \RR^d$ and finite $F \subseteq \ZZ^m$.

Set $H_0=\ZZ^m \times \KK$. It is easy to see that we can find function $\varphi \in \Cc^\infty(\RR^d)\cap K_2(\RR^d)$ and $\psi \in K_2(H_0)$ with the following properties:
\begin{itemize}
  \item{} $\phi:= \varphi \otimes \psi \equiv 1$ on $W$.
  \item{} $\widehat{\psi} \in \Cc(\widehat{H_0})$.
\end{itemize}
It follows that
\begin{equation}\label{eq8}
f:=\check{\phi}= \check{\varphi} \otimes \check{\psi}
\end{equation}
is a strongly admissible function of $\widehat{H} =\RR^d \times \widehat{H_0}$.

\smallskip

Next, as in Theorem~\ref{T1}, set $\eta:=\LI_{G, \RR^d \times H,\cL, W}(\gamma)$. Then, by Thm.~\ref{T1} $\eta$ is Fourier transformable. Moreover, since $\supp(\eta) \subseteq \cL$, $\rho=\widehat{\eta}$ is $\cL^0$-periodic \cite{ARMA}. Finally, by \eqref{eq6}, we get the following formula, which describes the Fourier transform of $\gamma$ as the projection in the dual CPS of a $\cL^0$-periodic measure via a strongly admissible function:
\begin{equation}\label{eq9}
\widehat{\gamma}=(\rho)_{f} \,.
\end{equation}
Fact~\ref{F2} then gives the following result.
\begin{corollary} \cite[Thm.~7.1]{NS21} Let $\gamma$ be a measure with Meyer set support. Then, $\widehat{\gamma}$ is norm almost periodic.
\end{corollary}

\subsection{Generalized Eberlein decomposition}

In this subsection we show a pseudo-compatibility of the mapping $\rho \to (\rho)_f$ of \eqref{eq4}, for $\cL$ periodic $\rho \in \cM_{\cL}(G \times H)$ and strongly admissible $f$, with respect to the Lebesgue decomposition. We explain this below.

First, it is easy to see that the map satisfies
\begin{itemize}
  \item{} if $\rho$ is pure point then $(\rho)_f$ is pure point;
  \item{} if $\rho$ is absolutely continuous then $(\rho)_f$ is absolutely continuous;
  \item{} if $\rho$ is singular continuous then $(\rho)_f$ can have all three spectral components;
\end{itemize}
and hence does not preserves the Lebesgue decomposition. On another hand, for each $\alpha \in \{\mathsf{pp}, \mathsf{ac}, \mathsf{sc}\}$, one can defined an operator
$P_\alpha$ on the space $\cM_{\cL}^\infty(G \times H)$ with the property that for all strongly admissible $f$ and all $\rho \in \cM_{\cL}^\infty(G \times H)$ we have
\begin{equation}\label{eq4}
\left( P_\alpha (\rho)\right)_{f}= \left((\rho)_{f}\right)_{\alpha} \,.
\end{equation}
This can be done simply by first showing that
$$
L_{\alpha}(\sum_{j=1}^m c_j \psi_j \odot \phi_j) := \sum_{j=1}^m c_j \left( \rho_{\phi_j} \right)_{\alpha}(\psi_j)
$$
for all $c_1,...,c_m \in \CC, \psi_1,..., \psi_m \in \Cc(\widehat{G}), \phi_1,..., \phi_m \in \Cc(\widehat{H})$ is well defined, linear and continuous with respect to the inductive topology. Therefore, $L_{\alpha}$ can be uniquely extended to a measure $P_\alpha (\rho)$, which is $\cL$ invariant and satisfies \eqref{eq4}.

\smallskip
Now, exactly as above let $\gamma$ be a Fourier transformable measure supported inside a Meyer set $\Lambda$ and let $(G,H, \cL), \eta, \varphi, \phi, \psi$ be as in Subsection~\ref{subs 1}. Let $f$ be as in \eqref{eq8} and let $\rho= \widehat{\eta}$.

Then, for each $\alpha \in \{\mathsf{pp}, \mathsf{ac}, \mathsf{sc}\}$, the measure $P_\alpha (\rho)$ is the Fourier transform of some measure $\mu$ supported on $\cL^0$ \cite{CRS3}.

Define
$$
\nu:= \sum_{x \in L} \phi(x^\star) \mu(\{ (x,x^\star) \}) \delta_x  \,.
$$
Then, $\supp(\nu) \subseteq \oplam(\supp(\phi))$ and, exactly as in the proof of Theorem~\ref{T1} we get
\begin{align*}
  \widehat{\nu} &= \left( \widehat{\mu} \right)_{h}= \left( P_\alpha (\widehat{\eta}) \right)_{h}  \\
  &=\left((\widehat{\eta})_{h}\right)_{\alpha}= \left( \widehat{\gamma} \right)_{\alpha} \,.
\end{align*}
Therefore, we get
\begin{corollary} \cite[Thm.~4.1]{NS21} Let $\gamma$ be a Fourier transformable measure supported inside a Meyer set $\Lambda$. Then, there exists a model set $\Gamma \supseteq \Lambda$ and three Fourier transformable measures $\gamma_{\mathsf{s}},\gamma_{\mathsf{0s}}, \gamma_{\mathsf{0a}}$ supported inside $\Gamma$ such that
\begin{align*}
  \reallywidehat{\gamma_{\mathsf{s}}} &=\left( \widehat{\gamma} \right)_{\mathsf{pp}}  \\
  \reallywidehat{\gamma_{\mathsf{0s}}} &=\left( \widehat{\gamma} \right)_{\mathsf{sc}}  \\
  \reallywidehat{\gamma_{\mathsf{0a}}} &=\left( \widehat{\gamma} \right)_{\mathsf{ac}}  \\
\end{align*}
\end{corollary}

\subsection{Discussion}

We have seen in this section that the Fourier transform of a measure $\gamma$ supported inside a Meyer set $\Lambda$ can be describe via \eqref{eq9} as the projection in the dual CPS of a $\cL^0$-periodic measure via a strongly admissible function. We used this result to (re)derive properties of $\widehat{\gamma}$, and we expect that this connection will lead to some new applications in the future. Indeed, while now we know quite a  few properties of the Fourier transform of measures with Meyer set support \cite{JBA,NS1,NS2,NS5,NS11,NS21,NS21b,NS21c} we know much more about fully periodic measures in LCAG (see for example \cite{CRS3}). Moreover, the strong admissibility of $f$
is likely to transfer many properties from $\rho$ to $\rho_f$. It is also worth pointing out that, while the strong admissibility of $f$ was sufficient to derive the conclusions in this section, in fact $f$ can be chosen of the form $f : \RR^d \times (\TT^m \times \widehat{\KK})$
$$
f := g \otimes P \otimes \psi :\RR^d \times \TT^m \times \widehat{\KK} \to \CC
$$
with $g \in \cS(\RR^d)$ being the Fourier transform of some $\varphi \in \Cc^\infty(\RR^d)$; $P$ being a trigonometric polynomial that is a sum of characters
$$
P= \sum_{j=1}^m \chi_j \qquad ; \qquad \chi_1,\ldots , \chi_j \in \widehat{\TT^m}
$$
and $\psi \in \Cc(\widehat{\KK})$ being the characteristic function of $\{ 0 \}$. This seems to be much stronger than strong admissibility and potentially of help in the future.

\subsection*{Acknowledgments} We are grateful to Michael Baake and Christoph Richard  for many insightful discussions which inspired this manuscript. This work was supported by NSERC with grant 2020-00038, and the author is grateful for support.

\end{document}